\documentclass[12pt,reqno, a4paper]{amsart}

\usepackage[T1]{fontenc}
\usepackage[english]{babel}

\usepackage[margin=1in,footskip=.2in]{geometry}\usepackage{amsmath,amssymb,amscd,amsfonts,mathtools,tikz,caption,float,tikz-cd}
\usetikzlibrary{patterns}
\usepackage[linkcolor=blue,colorlinks=true,citecolor=blue, urlcolor=blue]{hyperref}
\usepackage[capitalise]{cleveref}
\usepackage{mathptmx}
\usepackage{mathrsfs}
\usepackage{subfigure}
\usepackage{algorithm, algpseudocode}
\usepackage{comment}

\newtheorem{thm}{Theorem}[section]
\newtheorem{lem}[thm]{Lemma}
\newtheorem{defn}[thm]{Definition}
\newtheorem{prop}[thm]{Proposition}
\newtheorem{coro}[thm]{Corollary}
\newtheorem{rmk}[thm]{Remark}

\newtheorem{ex}[thm]{Example}

\numberwithin{equation}{section}
\counterwithin{algorithm}{section}

\newcommand{\norm}[1]{\left\Vert#1\right\Vert}
\newcommand{\abs}[1]{\left\vert#1\right\vert}

\newcommand{\R}{\mathbf{R}}

\newcommand{\C}{\mathbf{C}}

\newcommand{\Q}{\mathbf{Q}}
\newcommand{\Z}{\mathbf{Z}}

\newcommand{\PSL}{\mathrm{PSL}}

\newcommand{\Aut}{\mathrm{Aut}}

\newcommand{\sC}{\mathsf{C}}
\newcommand{\sS}{\mathsf{S}}

\newcommand{\cp}{\C P}

\title{Crossing the transcendental divide: from Schottky groups to algebraic curves}
\author{Samantha Fairchild, Ángel David Ríos Ortiz}

\begin{document}
\begin{abstract}
Though the uniformization theorem guarantees an equivalence of Riemann surfaces and smooth algebraic curves, moving between analytic and algebraic representations is inherently transcendental. Our analytic curves identify pairs of circles in the complex plane via free groups of M\"obius transformations called Schottky groups. We construct a family of non-hyperelliptic surfaces of genus $g\geq 3$ where we know the Riemann surface as well as properties of the canonical embedding, including a nontrivial symmetry group and a real structure with the maximal number of connected components (an $M$-curve). We then numerically approximate the algebraic curve and Riemann matrices underlying our family of Riemann surfaces.
\end{abstract}
\maketitle

\section{Introduction}\label{sec:introduction}

Riemann surfaces are ubiquitous objects in mathematics. As such, they come in different incarnations such as geometric, analytic, and algebraic. Each incarnation has its advantages and is suitable for different purposes. Although the passage from the geometric to the algebraic side is theoretically possible, in order to make this passage effective (i.e. with a computer) we need to choose the right geometric and algebraic presentations. We give steps in bridging the \emph{transcendental divide} between Riemann surfaces presented as quotients of Schottky groups, a presentation coming from Geometric Group Theory, and Riemann surfaces represented as canonical curves in $\cp^{g-1}$, their most natural presentation in Algebraic Geometry. 

Inspired by Bernd Sturmfels, and coined in \cite{CFM22}, the transcendental divide refers to the fact that explicitly or numerically connecting a Riemann surface to an algebraic curve requires transcendental functions, that is, functions which cannot be written as polynomials. The transcendental divide has been a large source of mathematics for the past century, inspired by work of Riemann, Poincar\'e, and Klein  \cite{farkas_kra, poincare}. A limiting aspect of working explicitly with the classical theory were the lack of computational tools, which have seen great advances in recent years \cite{BK13}. Particular work includes computational work on computing Riemann theta functions \cite{computing04}, which has now been implemented in two programming languages \cite{SD16, AC21}. 

We have three main contributions. Firstly, we construct an interesting family of non-hyperelliptic Riemann surfaces that admit an action of the dihedral group and their real structure attains the maximal number of ovals. Secondly, we present an algorithm that evaluates the canonical image of a Riemann surface given in terms of Schottky data. Finally, we show in low genus how our algorithm effectively captures the theoretical properties of our family of curves. This is showed in low genus by approximating the ideal of a curve and computing its Riemann matrix. One of the main features of our framework is its flexibility, paving the way of extending our work for higher genus and therein bringing together again Geometric Group Theory, Algebraic Geometry and Computational Mathematics. We now introduce some notation in order to describe our results.

We construct a genus $g\geq 2$ Riemann surface as follows. Given $2g$ circles $\sC_1,\ldots, \sC_g$ and $\sC_1', \ldots, \sC_g'$, where each pair of circles are disjoint circles in the Riemann sphere $\cp^1$ and $\sC_j, \sC_j'$ have the same radius (See \cref{fig:Schottkyex}), consider M\"obius transformations $f_1,\ldots, f_g$ which map the inside of $\sC_j$ to the outside of $\sC_j'$ for each $j=1,\ldots, g$. The \textit{(classical) Schottky group} $G$ is the free group generated by $f_1,\ldots, f_g$. Each generator $f_j$ has fixed points $A_j, B_j$, and the limit set $\Lambda$ is the union of orbits of $A_j,B_j$ for $j=1,\ldots, g$. The group $G$ acts freely and properly discontinuously on the domain of discontinuity $\Omega = \cp^1\setminus \Lambda$, and the genus $g$ Riemann surface is given by the quotient $\Omega/G$. The study of Schottky groups and other discrete subgroups of $\PSL(2,\C)$ is a beautiful and rich theory, with some accessible entry points from \cite{indras_pearls, marden}.

\begin{figure}
    \centering
    \begin{tikzpicture}[line cap=round,line join=round,x=1 cm,y=1 cm]
        \draw[dashed] (-1.5,0)--(1.5,0);
        \draw[dashed] (0,-1.5)--(0,1.5);
        \draw[fill =gray!30, draw = cyan] (.897,.518) circle (.268); 
        \node at (1.5,1) {$\sC_1$};
        \draw[fill =gray!30, draw = cyan] (.897,-.518) circle (.268); 
        \node at (1.5,-1) {$\sC_1'$};
        \draw[fill =gray!30, draw= purple](-.897, .518) circle (.268);  
        \node at (-1.5,1) {$\sC_2'$};
        \draw[fill =gray!30, draw= purple](-.897, -.518) circle (.268);  
        \node at (-1.5,-1) {$\sC_2$};

         \draw[->]  (1.2,.5) to [out = -30, in =30]  (1.2,-.5) ;  
        \node at (1.7,0) {$f_1$};
        \draw[->]  (-1.2,-.5) to [out = 150, in = -150]  (-1.2,.5) ;  
        \node at (-1.7,0) {$f_2$};
        
\draw[smooth]  (3,1) to[out=30,in=150] (5,1) to[out=-30,in=210] (6,1) to[out=30,in=150] (8,1) to[out=-30,in=30] (8,-1) to[out=210,in=-30] (6,-1) to[out=150,in=30] (5,-1) to[out=210,in=-30] (3,-1) to[out=150,in=-150] (3,1);
\draw[smooth] (3.4,0.1) .. controls (3.8,-0.25) and (4.2,-0.25) .. (4.6,0.1);
\draw[smooth] (3.5,0) .. controls (3.8,0.2) and (4.2,0.2) .. (4.5,0);
\draw[smooth] (6.4,0.1) .. controls (6.8,-0.25) and (7.2,-0.25) .. (7.6,0.1);
\draw[smooth] (6.5,0) .. controls (6.8,0.2) and (7.2,0.2) .. (7.5,0);
\draw[color = purple] (5.0,0.0) arc(0:360:1 and 1.14/2) node[above, pos = .25] {$\sC_2 = \sC_2'$};;
\draw[color=cyan] (8.0,0.0) arc(0:360:1 and 1.14/2) node[above, pos = .25] {$\sC_1 = \sC_1'$};;
    \end{tikzpicture}
    \caption{A genus 2 curve by gluing the circles $\sC_j$ to $\sC_j'$ for $j=1,2$ under M\"obius transformations $f_1,f_2$.}
    \label{fig:Schottkyex}
\end{figure}
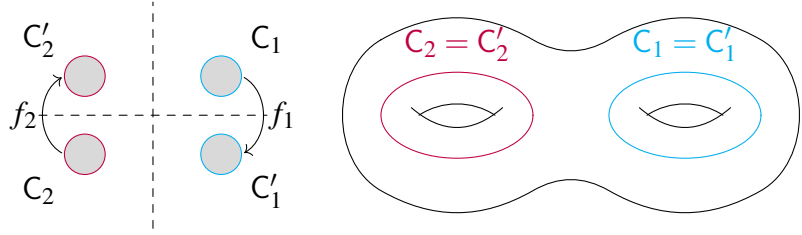

The key transcendental functions in this paper are \textit{Poincar\'e theta series} \cite[p. 148]{poincare}, which define a basis of Abelian differentials on the Riemann surface $S$.

\begin{prop}\label{prop:oneforms}\cite[Section~5.1]{BK13}
    The following Poincar\'e theta series, when convergent, form a basis of holomorphic differentials for $n=1,\ldots, g$
    $$\omega_n = \sum_{f \in G/G_n} \left(\frac{1}{z- f(B_n)} - \frac{1}{z - f(A_n) }\right)\,dz$$
    where $G/G_n$ are cosets with representatives given by all elements $f_{i_1}^{j_1}\cdots f_{i_k}^{j_k}$ where $i_k \neq n$ and $A_n, B_n$ are the fixed points of $f_n$.
\end{prop}
 By evaluating these series we obtain the \emph{canonical map} defined by
\begin{equation}
    \phi_S: S\longrightarrow \cp^{g-1}; \quad z\mapsto [\omega_1(z):\dots:\omega_g(z)].
\end{equation}

The benefit of studying the canonical maps is that the \emph{extrinsic} geometry of the Riemann surface in $\cp^{g-1}$ is reflected in the \emph{intrinsic} structure of the Riemann surface itself. For hyperelliptic Riemann surfaces, as is the case for genus two, the map $\phi_S$ is a $2:1$ map onto $\cp^1$, whereas for \emph{general} Riemann surfaces $\phi_S$ turns out to be an embedding. 

\subsection{A new family of examples}
We are not aware of any explicit equations for a single canonical curve coming from a Schottky group, but we include after the statement of our results some context of different techniques used for understanding the canonical curve for Schottky groups in general. The main result of this paper constructs a family of Schottky groups where the curves have a real structure and sufficiently many automorphisms to verify the computational efficacy of the approximate canonical curve. This theorem includes work of Hidalgo \cite{Hidalgo02} in the case of genus $3$, and extends the family of examples to arbitrary genus.

We now present the construction of our family of examples and refer to \cref{sec:Schottky} for the relevant notations. We denote $\sC(z_i,r)$ for the circle of radius $r$ centered at $z_i$.
\begin{defn} \label{def:family}
Fix a genus $g\geq 2$. Let $0<\eta < \theta <\frac{\pi}{g}$ be real numbers. Define a Riemann surface $S_{\eta,\theta}$ constructed from a Schottky group $G_{\eta,\theta}$  defined by identifying the circle pairs from the curves $\sC_i = \sC(z_i,r)$ to $\sC_i' = \sC(w_i,r)$ with positive real trace where $\sC_1$ is the curve orthogonal to the unit circle passing through $e^{i\eta}$ and $e^{i\theta}$, $\sC_1'$ is the conjugate of $\sC_1$, and the remaining $\sC_i$ are given by rotations of $\frac{2\pi i}{g}$ (see \cref{fig:g23} for $g=3$).
\end{defn}

To state the properties of $S_{\eta,\theta}$, we need the following definitions. A \textit{real structure} on a Riemann surface $X$ is an anti-holomorphic involution $\tau$. A real surface $X$ is called an \textit{$M$-curve} if it it attains the maximum number of ovals fixed under $\tau$, namely $g+1$. We will use the geometric convention that $D_g$ represents the dihedral group of order $2g$. Recall that the signature of a group $H$ acting holomorphically on a Riemann surface $X$ is the tuple $(g_0;m_1,\dots,m_2)$ where $g_0$ is the genus of the quotient $X/H$ and the map $X\to X/H$ is branched over $s$ points with orders $m_1,\dots,m_s$.

\begin{thm}\label{thm:family}
 The Riemann surface $S:= S_{\eta,\theta}$ is an $M$-curve with real structure defined by inversion about the unit circle $z\mapsto 1/\bar{z}$, which is invariant under the action of the Schottky group $G:= G_{\eta,\theta}$. The group of automorphisms $\Aut(S)$ contains the dihedral group $D_g$, which acts on $S$ with signature $(0;2,2,2,2,g)$. 
\end{thm}
    We have the following results split into the hyperelliptic case $g=2$, and the non-hyperelliptic case $g\geq 3$. Recall that a Weierstrass point in a Riemann surface $X$ is a point $P$ for which there exists a meromorphic function with a single double pole at $P$.
\begin{thm} \label{thm:maing2}
    For $g=2$, the Riemann surface $S_{\eta,\theta}$ is represented by $y^2 = \prod_{j=1}^6 (x- \tilde{\phi}_S(x_j)),$ where the $x_j$ are the Weierstrass points
$$x_1 =1, \quad x_2 =-1, \quad x_3 = e^{i\theta},\quad x_4 =  e^{i\eta},\quad x_5 = e^{i(\pi-\theta)},\quad x_6 =e^{i(\pi-\eta)}$$
and $\tilde{\phi}(z) = \omega_2(z)/\omega_1(z)$ is the coordinate chart where $\omega_1(z)\neq 0$, and the images $\tilde{\phi}(x_j) $ are real and given in reciprocal pairs.
\end{thm}

See \cref{fig:g23} for reference of the Weierstrass points in \cref{thm:maing2}.
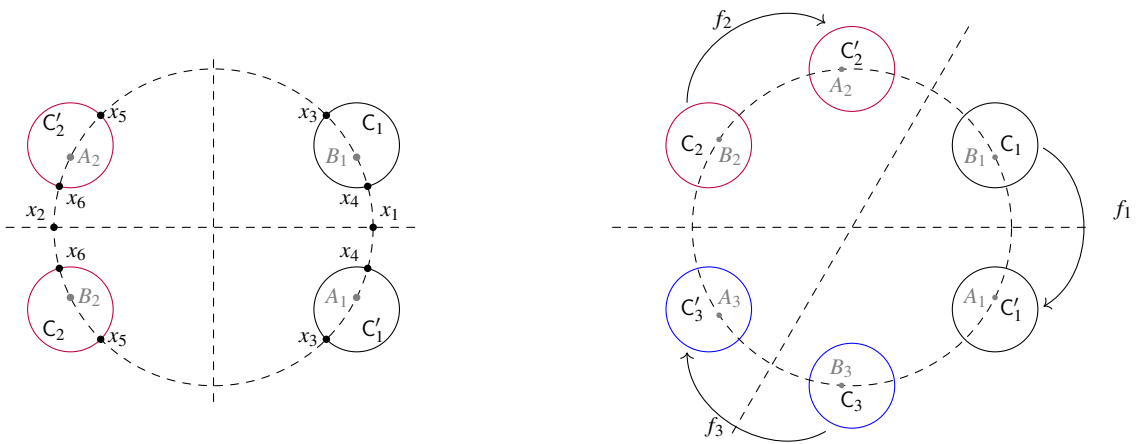
\begin{figure}[hb]
    \centering
    \begin{tikzpicture}[line cap=round,line join=round,x=3 cm,y=3 cm, scale=0.7, every node/.style={scale=0.7}]
        \draw[dashed] (0,0) circle (1);
        \draw[dashed] (-1.3,0)--(1.3,0);
        \draw[dashed] (0,-1.1)--(0,1.1);
        \draw (.897,.518) circle (.268); 
        \node at (1,.65) {$\sC_1$};
        \draw (.897,-.518) circle (.268); 
        \node at (1,-.65) {$\sC_1'$};
        \draw[color= purple](-.897, .518) circle (.268);  
        \node at (-1,.65) {$\sC_2'$};
        \draw[color= purple](-.897, -.518) circle (.268);  
        \node at (-1,-.65) {$\sC_2$};
         \fill (1,0)  circle[radius=2pt] node[above right]{$x_1$};
         \fill (-1,0)  circle[radius=2pt] node[above left]{$x_2$};
        
         \fill (.707, 0.707)  circle[radius=2pt] node[left]{$x_3$};
         \fill (.707, -0.707)  circle[radius=2pt] node[left]{$x_3$};

         \fill (.966,.259)  circle[radius=2pt] node[below left]{$x_4$};
         \fill (.966,-.259)  circle[radius=2pt] node[above left]{$x_4$};

         \fill (-.707, 0.707)  circle[radius=2pt] node[right]{$x_5$};
         \fill (-.707, -0.707)  circle[radius=2pt] node[right]{$x_5$};

         \fill (-.966,.259)  circle[radius=2pt] node[below right]{$x_6$};
          \fill (-.966,-.259)  circle[radius=2pt] node[above right]{$x_6$};

         \fill[gray] ( 0.8965754721680534, - 0.44289098286895795)circle[radius=2pt] node[left]{$A_1$};
        \fill[gray] (0.8965754721680537 , 0.44289098286895806) circle[radius=2pt] node[left]{$B_1$};
        \fill[gray] (-0.8965754721680537, 0.44289098286895806) circle[radius=2pt] node[right]{$A_2$};
        \fill[gray] (-0.8965754721680534, - 0.44289098286895795) circle[radius=2pt] node[right]{$B_2$};

        \draw[dashed] (0+4,0) circle (1);
        \draw[dashed] (-1.5+4,0)--(1.5+4,0);
        \draw[dashed] (-.75+4, -1.299)-- (.75+4, 1.299);
        \draw (.897+4,.518) circle (.268); 
        \node at (1+4,.518) {$\sC_1$};
        \draw (.897+4,-.518) circle (.268); 
        \node at (1+4,-.518) {$\sC_1'$};
        \draw[->]  (1.2+4,.5) to [out = -30, in =30]  (1.2+4,-.5) ;  
        \node [anchor=south] at (1.7+4,0) {$f_1$};

        \draw[color= purple](-.897+4, .518) circle (.268);  
        \node at (-1+4,.518) {$\sC_2$};
        \draw[color= purple](0+4,1) circle (.268);  
         \node at (0+4,1.1) {$\sC_2'$};

        \draw[color= blue](0+4, -1) circle (.268);  
        \node at (0+4,-1.1) {$\sC_3$};
        \draw[color= blue](-.897+4, -.518) circle (.268);  
         \node at (-1+4,-.518) {$\sC_3'$};
        \draw[->]  (-1.033+4,.789) to [out = 90, in = 150]  (-.166+4,1.289) ;  
        \node [anchor = south east] at (-.7+4,1.2) {$f_2$};
        \draw[->] (-.166+4,-1.289)  to [out = -150, in = -90]  (-1.033+4,-.789) ; 
        \node [anchor = north east] at (-.75+4,-1.15) {$f_3$};

        \fill[gray] ( 0.8965754721680534+4, - 0.44289098286895795)circle[radius=1.5pt] node[left]{$A_1$};
        \fill[gray] ( -0.06473289381245049+4, 0.9979026267420412)circle[radius=1.5pt] node[below]{$A_2$};
        \fill[gray] ( -0.831842578355+4, -0.5550116438730832)circle[radius=1.5pt] node[above]{$\quad A_3$};
        
        \fill[gray] (0.8965754721680537+4 , 0.44289098286895806) circle[radius=1.5pt] node[left]{$B_1$};
        \fill[gray] ( -0.831842578355+4, 0.5550116438730832) circle[radius=1.5pt] node[below]{$\quad B_2$};
        \fill[gray] ( -0.06473289381245049+4, -0.9979026267420412) circle[radius=1.5pt] node[above]{$B_3$};       
    \end{tikzpicture}

    \caption{Examples of the curves from \cref{def:family} for $g=2,3$ on the left and right respectively with fixed points $A_j$ and $B_j$ the attracting (resp.~repelling) fixed points of $f_j$. Also on the left see \cref{thm:maing2} for the Weierstrass points $x_i$.}
    \label{fig:g23}
\end{figure}

\begin{thm}\label{thm:maing3} For $g\geq 3$, the Riemann surface $S := S_{\eta,\theta}$ is not hyperelliptic. Hence the canonical map $\phi_S:S\to\cp^{g-1}$ is an embedding.
\end{thm}

In the case of low genus, we can use \cref{thm:maing3} to give the exact form of the curve in terms of fixed points of $\phi_S$. In \cref{sec:algebraiccurves} we explore these symmetries for genus $3$ and $4$.

Riemann surfaces of the same genus can be grouped in a moduli space $\mathcal{M}_g$. Roughly speaking, the moduli space of genus $g$ Riemann surface is a complex manifold with the property that every family of Riemann surfaces of genus $g$ has a unique map to $\mathcal{M}_g$. There are two ways of doing this: One analytical (mapping class groups) and one algebraic (stacks). See \cite{ACGVolII} for more details. There are many interesting questions that arise from studying the geometry of the family of Riemann surfaces from \cref{def:family} that are beyond the scope of this paper, but we include a short discussion in \cref{sec:finalmoduli} for those interested. We highlight that an important benefit of \cref{def:family} is that the range of values for $\theta$ and $\eta$ provides a continuous, $2$-dimensional, real parametrization for a family in $\mathcal{M}_g$.

\subsection{Numerical approximations}
In \cref{sec:NumericalExperiments}, we use \cref{def:family} as our test case to numerically evaluate the canonical map by taking finite approximations of the series. In genus 2 (\cref{sec:g2numerical}) we approximate the Weierstrass points, which are the fixed points under the hyperelliptic involution. In genus $\geq 3$ we will not work with hyperelliptic curves, and by sampling a distribution of points on $S$, approximate the equations of the algebraic curve (\cref{sec:g3numerical}).

To do the numerical approximations, we give \cref{alg:circ2schottkywords} as a convenient source for those hoping to implement the techniques of these approximations for examples outside of the scope of this paper. Given a a list of isometric circles, we construct Schottky group elements identifying the circles and construct the elements of the Schottky group of bounded word length. This uses classical properties of M\"obius transformations and the Breadth First search presented in \cite[p.115]{indras_pearls} in updated language. The full code was implemented in Julia \cite{julia} and source code can be found in \cite{mathrepo}.

In \cref{sec:numericalRiemann} we give the formula for approximating the Riemann matrix associated, and compute a specific example for genus $3$. Recall a Riemann matrix is a $g\times g$ matrix with positive definite imaginary part. For $M$-curves, we can in fact always represent the Riemann matrix as a purely imaginary matrix. The power of $M$-curves arise in the study of solutions to the Kadomtsev--Petviashvili (KP) partial differential equation which is used to describe nonlinear wave motion . In fact $M$-curves are all of the nonsingular integrable (specifically finite-gap) solutions of the KP equation \cite{Bobenkofinitegap}. In addition to the computational work surrounding $M$-curves contained in \cite{BK13}, see also more recent work on $M$-curves for example in mathematical physics \cite{AG22} and algebraic geometry \cite{KS22}.

\subsection{Connection to other works} 

When $g=2$, a Riemann surface is always hyperelliptic, in which case the algorithms are not necessarily the most efficient.  In the case of real hyperelliptic surfaces, there exist more efficient algorithms for evaluating the series in \cref{prop:oneforms} due to S. Yu. Lyamaev \cite{Lyamaev}. For those needing more efficient computations than those presented in this paper, we recommend beginning with extremal polynomials \cite{Bogatyrev}. There is is a stronger version of hyperellipticity coming from hyperelliptic Schottky groups, where hyperelliptic structure of the curve can be seen in the group structure as well, see \cite{Keen80} and more recently in \cite{KeenGilman05}.

Staying in the genus 2 setting, in \cite{HidalgoFigueroa2005}, they provide an example where they solve the inverse uniformization theorem by (numerically) constructing a Schottky group that normalizes Bolza's curve. In addition to being of independent interest for numerical results related to Schottky groups on genus $2$ curves, the numerical results of \cite{HidalgoFigueroa2005} are particularly useful to provide calibration of our code. We remark however that in \cite{HidalgoFigueroa2005}, they consider a slightly different series, and to calibrate our results, we can use the fact that our series match at $-1$ and $1/\sqrt{3}$ for $F_\frac{\pi}{6}$. Other examples with in the hyperelliptic setting include a result on constructing Schottky groups and associated algebraic curves for any genus $g\geq 2$ for a family of hyperelliptic curves called May's surfaces in \cite{Hidalgo02}.  

We refer the reader to \cite{Hidalgo2005} for more information on the study of automorphism groups for Shottky groups, and to \cite{HidalgoRealSchottky, BHQ18Real, Hidalgo24} for further discussions on the connections between Schottky groups and real structures in the hyperelliptic and non-hyperelliptic cases. We will focus on curves over $\C$, but there are rich analogous theories over other field. Over $p$-adics, one can also consider Schottky groups, for which the images are then called Mumford Curves \cite{Mumford1972}, see \cite[Page 37]{GerritzenvanderPut} and the recent work of J. Poineau and D. Turchetti, see \cite{PoineauTurchetti2021,PoineauTurchetti2021II,PoineauTurchetti2022}. Over algebraic numbers, the theory of Dessins d'Enfants gives a way to construct such algebraic curves via Beyl\u{i}'s theorem \cite{jwdessins}.

\subsection{Outline} We begin the paper with Preliminaries in \cref{sec:prelim} covering the necessary background in Algebraic curves, Schottky groups, as well as necessary conditions for convergence of the series in \cref{prop:oneforms}. In \cref{sec:maintheorems} we prove the main theorems of the paper: \cref{thm:family}, \cref{thm:maing2} and \cref{thm:maing3}. In \cref{sec:algebraiccurves} we include the computations needed to obtain explicit algebraic curves, which are then used in the numerical experiments in \cref{sec:NumericalExperiments}. We conclude with some final remarks and future directions in \cref{sec:finaldiscussion}.

\subsection*{Acknowledgements} The authors would like to thank Bernd Sturmfels for connecting us on this problem and Alex Elzenaar for contributions at the beginning of this project. We would also like to thank Benjamin Hollering and T\"urk\"u Ozl\"um \c{C}elik for useful discussions. Fairchild was partially supported by Deutsche Forschungsgemeinschaft (DFG) - Project number 507303619. Rios Ortiz was supported by the European Research Council (ERC) under the European Union’s Horizon 2020 research and innovation programme (ERC-2020-SyG-854361-HyperK).

\section{Preliminaries}\label{sec:prelim}

In this section we provide some background and references for more detailed analysis on Algebraic curves (\cref{sec:prelimcurves}) and Schottky groups (\cref{sec:Schottky}).

\subsection{From Riemann surfaces to algebraic curves} \label{sec:prelimcurves}
Every compact Riemann surface $S$ can be embedded in the projective space $\cp^n$ for large enough $n$ as a complex \emph{algebraic} submanifold. In other words, there exists a homogeneous ideal $I\subset \C[z_0,\dots,z_n]$ such that $S$ is biholomorphic to the set of common zeroes of $I$ in $\cp^n$.

\begin{ex}
    Any Riemann surface $S$ of genus $1$ can be embedded in $\cp^2$  as a plane cubic curve, that is the zero set of a homogeneous polynomial of degree $3$ in the variables $z_0,z_1,z_2$. After fixing a point they inherit a group law, these are also called \emph{elliptic curves}.
\end{ex}

Once we have an embedding of a Riemann surface $S$ in $\cp^n$, if $n>3$ we can always isomorphically project $S$ onto $\cp^{n-1}$. This shows that every Riemann surface can be embedded as an algebraic submanifold of $\cp^3$, cf. \cite[pp. 215]{GriffithsHarris1978}. However, and this will be important later, we prefer embeddings that are \emph{intrinsic}. In the following we will describe such embeddings.

Let $S$ be a compact Riemann surface of genus $g$. The \emph{canonical bundle} $\Omega_S$ of $S$ is the the dual of the holomorphic tangent bundle on $S$. Abelian differentials on $S$ are global sections of $\Omega_S$. Locally any Abelian differential can be written as $f(z)\mathrm{d}z$, where $f(z)$ is a holomorphic function. Denote by $H^0(S,\Omega_S)$ the vector space of Abelian differentials on $S$, then a classical theorem due to Riemann states that $H^0(S,\Omega_S)$ has dimension $g$. Let $\omega_1,\omega_2,\dots,\omega_g$ be a basis of $H^0(S,\Omega_S)$. Pointwise evaluation induces the so-called \emph{canonical map}:
\begin{equation}
    \phi_S: S\dashrightarrow \cp^{g-1}, \quad z\mapsto [\omega_1(z):\dots:\omega_g(z)].
\end{equation}
This map is unique up to a change of basis. Notice that in principle $\phi_S$ might not be defined everywhere, meaning that there could be points in $S$ which are evaluated to zero for all elements in the basis, Theorem \ref{thm:canonicalmapisanembedding} below excludes that case and moreover, says that $\phi_S$ is an embedding if and only if there exists no meromorphic function on $S$ having exactly two poles, this last condition means that $S$ is \emph{hyperelliptic}.

\begin{defn}
    A compact Riemann surface $S$ is called \emph{hyperelliptic} if $S$ can be expressed as a two-sheeted branched covering of $\cp^1$. The induced covering involution on $S$ is called the \emph{hyperelliptic involution}.  
\end{defn}

\begin{thm}\cite[III.10.2]{farkas_kra}\label{thm:canonicalmapisanembedding}
    Let $S$ be a compact Riemann surface of genus $g>1$, then the canonical map is defined everywhere. The canonical map is an embedding if and only if $S$ is not hyperelliptic. The degree of $S$ under the canonical map is $2g-2$.
\end{thm}

\begin{rmk}
The canonical map characterizes non degenerate smooth complex Riemann surfaces of genus $g$ and degree $2g-2$ in $\cp^{g-1}$, see \cite[pp.139]{farkas_kra}.
\end{rmk}
To study the excluded case in Theorem \ref{thm:canonicalmapisanembedding}, let $f:S\to\cp^1$ be a hyperelliptic Riemann surface. By the Riemann-Hurwitz formula \cite[pp. 216]{GriffithsHarris1978} $f$ has $2g+2$ ramification points (called \emph{Weierstrass points}) and hence $2g +2$ branch points $z_1,\dots,z_{2g+2}$ in $\cp^1$. Therefore $S$ is the Riemann surface associated to the algebraic function
\begin{equation}\label{eq:AssRiemannSurfaceHyperelliptic}
    w^2 = (z-z_1)\cdots(z-z_{2g+2}).
\end{equation}


With the representation given in \ref{eq:AssRiemannSurfaceHyperelliptic} an explicit basis (see \cite[pp. 253]{GriffithsHarris1978} or \cite[III.10.3]{farkas_kra} for more details) for $H^0(S,\Omega_S)$ is given by the following differentials $\omega_1 = \frac{\mathrm{d}z}{w},\omega_2 = z\frac{\mathrm{d}z}{w}, \dots, \omega_g = z^{g-1}\frac{\mathrm{d}z}{w}$. The canonical map therefore factorizes in the following way
\begin{equation}
\begin{tikzcd}
S\ar[rr,"\phi_S"]\ar[rd,swap,"f"] & & \cp^{g-1}\\
 & \cp^1\ar[ur,swap,"\nu_{g-1}"]
\end{tikzcd}
\end{equation}
where $\nu_{g-1}$ is the \emph{Veronese embedding} of degree $g-1$, namely the map given in projective coordinates by $\nu_{g-1}([z:w]) = [z^{g-1}:z^{g-2}w:\cdots:w^{g-1}]$.
Every Riemann surface of genus $g\leq 2$ is hyperelliptic. Representations of hyperelliptic Riemann surfaces can be made very explicit using Weierstrass models, that is equations of the form \ref{eq:AssRiemannSurfaceHyperelliptic}, see \cite[VII.4]{farkas_kra} or \cite{Bogatyrev} for a more computational approach. Although our main interest in this work are non-hyperelliptic surfaces, we will deal with hyperelliptic curves in \ref{thm:maing2}.

Starting from genus $g\geq 3$ the \emph{general} Riemann surface is not hyperelliptic and therefore the canonical map is an embedding. We are interested in the \emph{equations} defining this embedding. In genus $3$ and $4$ non-hyperelliptic Riemann surfaces will be \emph{complete intersections}.

\begin{lem}\label{lem:g34ideal}\cite[pp. 256-259]{GriffithsHarris1978}
Let $S$ be a compact Riemann surface of genus $g$. Assume that $S$ is not hyperelliptic so $\phi_S$ is an embedding. If $g=3$, then the ideal $I_{\phi_S(S)}$ is generated by a unique polynomial of degree $4$. If $g=4$, then $I_{\phi_S(S)}$ is generated by two polynomials of degree $2$ and $3$ respectively.
\end{lem}

\subsection{Geometry of Schottky groups}\label{sec:Schottky}
\emph{Kleinian} groups are defined as discrete subgroups $G \leq \PSL(2,\C) = \mathrm{Aut}(\cp^1)$. In particular every Schottky group (see Introduction) is Kleinian. Recall, in \cref{sec:introduction} we defined the region $\Omega = \cp^1\setminus \Lambda(G)$ to be the \emph{domain of discontinuity}; where $\Lambda:= \Lambda(G)$ is the limit set. Every Riemann surface of genus at least 2 can in fact be realized as a Kleinian group.

\begin{thm}[Kleinian uniformization]{\cite[Section~VIII.B]{maskit} }\label{prop:uniform_klein}
Let $S$ be a Riemann surface of genus at least $2$. There exists a planar open set $\Omega$ and a group $G$ of M\"obius transforms which preserves $\Omega$, such that $S$ is biholomorphic to $\Omega/G$.
\end{thm}

We refer the reader to \cref{fig:limitsets} for a visualization of different limit sets for Schottky groups. One can generally construct Schottky groups by identifying pairs of isometric circles. On the left of \cref{fig:limitsets} is the Fuchsian case, where all circles have centers on a common circle or line, plotted in the case of \cref{def:family} for $g=2$, $\theta = \frac{5\pi}{12}$, $\eta = \frac{\pi}{12}$. In the middle is a classical Schottky group where the circles are disjoint but in more general positions. 
If the circles are allowed to overlap, there is still a notion of Schottky group, where the limit set becomes more interesting as shown in the right of \cref{fig:limitsets}. 

\begin{figure}[hbt]
    \includegraphics[width=.33\textwidth]{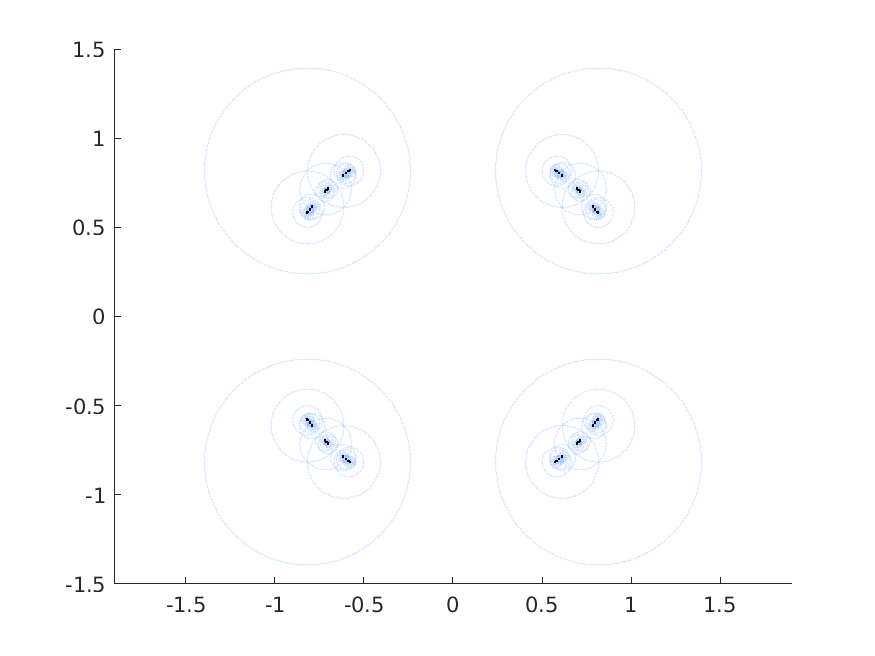}\hfill
    \includegraphics[width=.33\textwidth]{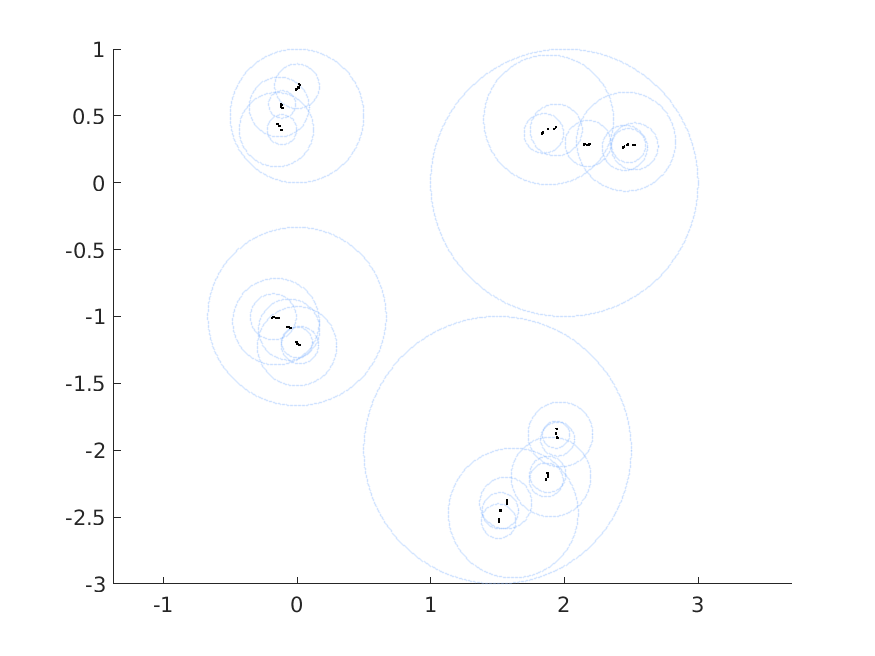}\hfill
    \includegraphics[width=.33\textwidth]{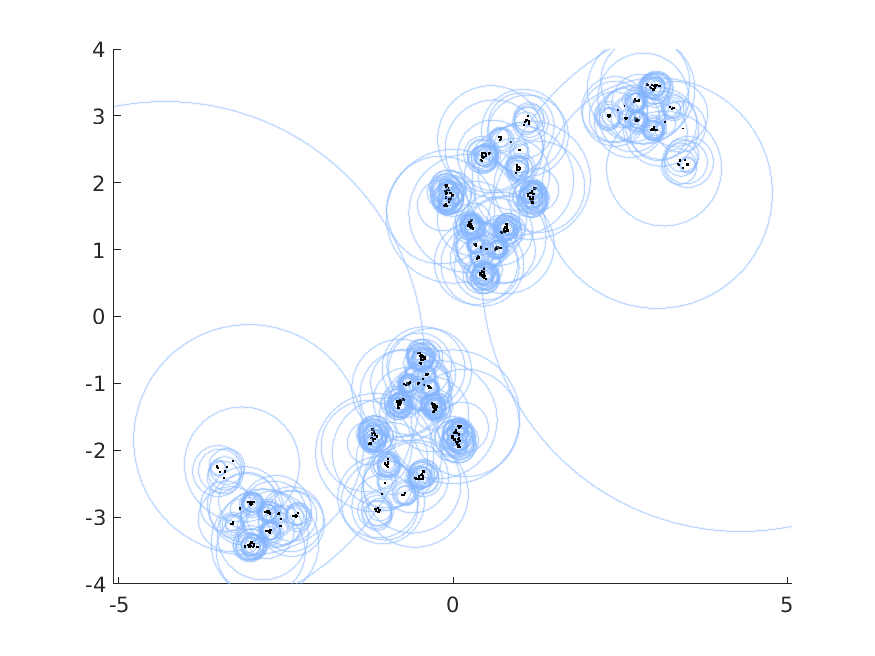}\hfill
    \caption{From left to right, three different limits sets $\Lambda$ for Fuchsian, classical, and general Schottky groups of genus 2.  We represent $\Lambda$ in dark blue and the light blue circles are images of the isometric circles under the elements in $G$.}\label{fig:limitsets}
\end{figure}

We now restrict ourselves to the case of Fuchsian Schottky groups introduced in \cref{sec:introduction} In this case the limit set is a discrete subgroup of the unit circle. This comes from the fact that the Schottky group leaves the unit circle unchanged so the fixed points, and all orbits of the fixed points reside on the circle. Such a group is called \emph{Fuchsian}. In general the complement of $\Omega(G)$ can be rather wild as shown in \cref{fig:limitsets}. We have the following result which is a special case of the results stated in \cite[\S12.2]{Ratcliffe} or \cite[\S II.G,IV.J]{maskit}.
\begin{prop}\label{prop:schottkycondition}
   Let $G$ be a Schottky group with generators $f_1,\ldots, f_g$ pairing disjoint circles with disjoint interiors $\sC_1,\ldots, \sC_g$ with $\sC_1', \ldots, \sC_g'$ so that $f_i(\sC_i) = \sC_i'$ and the image of the region bounded by $\sC_i$ under $f_i$ is the region outside of $\sC_i'$.
   \begin{enumerate}
       \item $G$ is free on $f_1,\ldots, f_g$
       \item $G$ is discrete
       \item The common exterior $U$ defined to be the complement of the $2n$ regions bounded by the $\sC_i, \sC_i'$ forms a fundamental domain for $G$.
       \item $\Omega/G$ is a compact surface of genus $g$.
   \end{enumerate}
\end{prop}
The proof requires the Poincare Polyhedron theorem \cite[\S11.2]{Ratcliffe} or \cite[\S IV.H]{marden}, which we use in \cref{prop:dihedralaction}.

We conclude this section by discussing the series in \cref{prop:oneforms}.  Their convergence is summarized in \cite[\S6.1.2]{Bogatyrev}, by making use of the so-called \emph{Schottky criterion}, which gives a sufficient condition for the series to converge absolutely and uniformly on compact subsets of $\Omega$, see \cite{Bogatyrev, Bobenkofinitegap}.

\section{Proof of main theorems}\label{sec:maintheorems}
In this section we will prove \cref{thm:family}, \cref{thm:maing2} and \cref{thm:maing3}. The idea is to generalize a construction by Hidalgo \cite[\S 6]{Hidalgo02} in order to get a Fuchsian Schottky group with many symmetries. Fix an integer $g>1$ and real numbers $\eta,\theta$ such that $ 0 < \eta < \theta < \pi/g $. Let $ L \subset \C $ be the line of slope $\pi/g$ through the origin and recall $\sC_1$ is the circle orthogonal to the unit circle $\sS^1\subset\C$ through the points $e^{i\theta}$ and $e^{i\eta}$; these objects are shown in \cref{fig:g23}. Now define $\sigma$, $\sigma_1$, $\sigma_2$, and $\sigma_3$ to be respectively the reflections across $\sS^1$, the real line $\R\subset\C$, $L$, and $\sC_1$ (see \cref{fig:Dgsigmas})

Construct two conformal mappings $r =\sigma_2\circ\sigma_1$ and $h =\sigma\circ\sigma_1$; these are respectively the clockwise rotation about $0$ by the angle $2\pi/g$ written as $z\mapsto e^\frac{2\pi i}{g} z$ and the order $2$ automorphism with fixed points $\pm 1$ given by $z\mapsto 1/z$.

For $k\in\Z$, define the element $f_k:= r^k \circ\sigma_1\circ\sigma_3 \circ r^{-k},$ and use these to generate the group
\begin{displaymath}
  G := \langle f_k : 0 \leq k \leq g-1 \rangle.
\end{displaymath}

Let $S:= \Omega(G)/G$ be the Riemann surface of genus $g$ obtained as a quotient. We can now prove the last part in \ref{thm:family}.

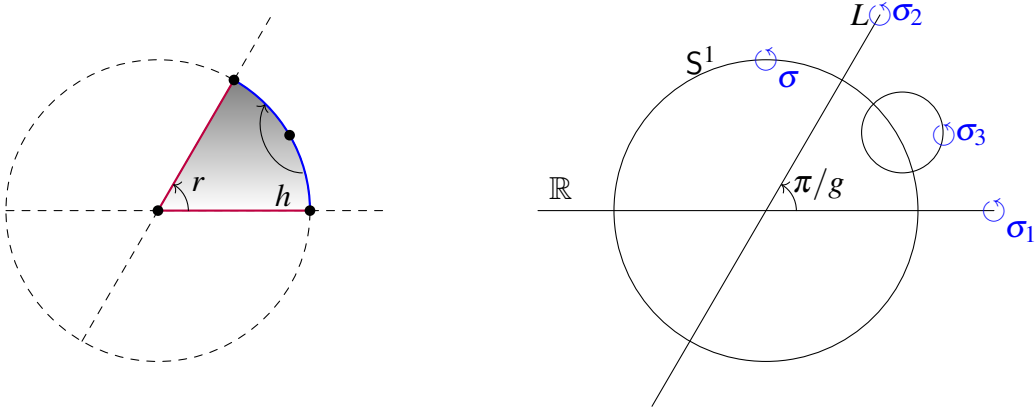
\begin{figure}[htb]
    \centering
    \begin{tikzpicture}[ line cap=round,line join=round,x=2 cm,y=2 cm]
        \shade (0,0) --(1,0)arc(0:60:1)--(.5, .866)--(0,0);
        \draw[dashed] (0,0) circle (1);
        \draw[dashed] (-1,0)--(1.5,0);
        \draw[dashed] (-.5, -.866)-- (.75, 1.299);
        
        \draw[purple, thick] (0,0)--(1,0);
        \draw[purple, thick] (0,0)--(.5,.866);
        \draw[->] (.2,0) arc (0:60:.2) node[anchor=west] {$\; r$};
        \draw[->] (.95,.25) node[anchor=north east] {$h$}to [out = 180, in =-120]  (.707,.707)  ;

        \draw[blue, thick] (1,0) arc (0:60:1);
        \fill (0,0) circle[radius=2pt];
        \fill (1,0) circle[radius=2pt];
        \fill (.866, .5) circle[radius=2pt];
        \fill (.5,.866) circle[radius=2pt];

         
         \draw (.897+4,.518) circle (.268);
      \draw[-] (-1.5+4,0) node[anchor = south west] {$\mathbb{R}$}--(1.5+4,0);
      \draw[-] (4,0) circle (1);
      \draw[-] (-.75+4, -.75*1.73)--(.75+4, .75*1.73) node[anchor = east]{$L$};
      \draw[->] (.2+4,0) arc (0:60:.2) node[anchor=west] {$\pi/g$};
      \node[left] at (-.25+4,1) {$\sS^1$};
      \node at (4,1) {\color{blue}$\circlearrowleft$} node [anchor= north west] at (4,1){\color{blue}$\sigma$};
      \node at (5.5,0) {\color{blue}$\circlearrowleft$} node [anchor= north west] at (5.5,0){\color{blue}$\sigma_1$};
      
      \node at (1.17+4,.5){\color{blue}$\circlearrowleft$} node [anchor= west] at (1.17+4,.5){\color{blue}$\sigma_3$};

      \node at (.75+4, .75*1.73){\color{blue}$\circlearrowleft$} node [anchor= west] at (.75+4, .75*1.73){\color{blue}$\sigma_2$};
        
    \end{tikzpicture}
    \caption{On the left is a fundamental domain for the action of the group $D_g = \langle r,h \rangle$. On the right, the transformations $f_j$ are formed by the reflections $\sigma, \sigma_1, \sigma_2, \sigma_3$ defined in \cref{sec:maintheorems}.}
    \label{fig:Dgsigmas}
\end{figure}

\begin{prop}\label{prop:dihedralaction}
With the notation as above, $ \langle r, h \rangle \simeq D_g $, the dihedral group of order $2g$, and also $D_g\subseteq \Aut(G)$. The signature of $D_g$ acting on $S$ is $(0;2,2,2,2,g)$.
\end{prop}

\begin{proof}
One checks readily that a fundamental domain for $\langle r, h \rangle$ is given by the sector of the unit disc cut out by $\R$ and $L$, in the positive quadrant. One views this as having four arcs as edges, paired by $r$ and $h$ as in \cref{fig:Dgsigmas}. With this we can apply the Poincar\'e polyhedron theorem (cf. \cite[\S IV.F.5]{maskit}) which implies that a
presentation for the group $\langle r,h \rangle $ has exactly the same relations as the dihedral group of order $2g$ :

Since $r$ and $h$ are M\"obius transformations, they will induce an automorphism on $S$ if conjugation by them preserves $G$. That is, we need to check that $ r f_k r^{-1} $ and $ h f_k h^{-1} $ lie in $G$ for all $k$. Indeed, we have $r f_k r^{-1} = r r^k \sigma_1 \sigma_3 r^{-k} r^{-1} = f_{k+1} \in G $, and
\begin{align*}
  h f_k h^{-1} &= h(r^k \sigma_1 \sigma_3 r^{-k}) h^{-1}\\
               &= \sigma \sigma_1 (\sigma_2 \sigma_1)^k \sigma_1 \sigma_3 (\sigma_1\sigma_2)^k \sigma_1 \sigma\\
               &= (\sigma \sigma) (\sigma_1 \sigma_2)^k (\sigma_1 \sigma_1) \sigma_3 \sigma_1 (\sigma_2 \sigma_1)^k\\
               &= g^{-k} \sigma_3\sigma_1 g^{k}\\
               &= \left(r^{g-k} \sigma_1\sigma_3 r^{-(g-k)}\right)^{-1}\\
               &= f_{g-k}^{-1} \in G.
\end{align*}
Hence the group $ \langle r,h \rangle = D_g$ is a subgroup of $\Aut(S)$. To compute its signature, notice that the orbits under the group action are the images under the quotient map $\Omega(G)\to S$ of the following points:
\begin{enumerate}
    \item $\{0,\infty\}$
    \item $e^{2k\pi i/g}$ for $k=1,\dots,g$
    \item $e^{(2k+1)k\pi i/2g}$ for $k=1,\dots,g$
    \item $e^{2ki\theta/2g}$ for $k=1,\dots,g$
    \item $e^{2ki\eta/2g}$ for $k=1,\dots,g$
\end{enumerate}
This implies the branch points of the quotient $S\to S/D_g$. To compute the genus {$g_0$} of $S/D_g$ we apply the Riemann--Hurwitz formula:
\[
g-1 = |D_g|(g_0-1 + \frac{1}{2}\sum_{i=1}^s(1-\frac{1}{m_i}) = 2g(g_0-1+\frac{1}{2}(5 - 2 - \frac{1}{g})=2gg_0+g-1.
\]
This implies that $g_0=0$.
\end{proof}

 We compute directly the action of the subgroup $D_g = \langle r,h\rangle \leq \Aut(S)$. 
 \begin{lem} The action of $r$ and $h$ satisfy $r^* \omega_j = \omega_{j+1}$ and $h^* \omega_j = -\omega_{g-j}$.
\end{lem}
\begin{proof}
    This follows by a direct computation.
\end{proof}
\begin{coro}\label{lem:-zmap}
    The map $z\mapsto -z$ acts nontrivially as a reciprocal under $\phi_S$ in genus 2 in the sense that $\tilde{\phi}_S(-z) \cdot \tilde{\phi}_S(z) = 1$.
\end{coro}
\begin{proof}
   The rotation matrix acts by $z\mapsto -z$ in the case of genus $2$.
\end{proof}

We now show that $h$ is the hyperelliptic involution in the case of $g=2$.

\begin{proof}[Proof of \cref{thm:maing2} Weierstrass points]
    Notice that $h(z) =1/z$ has fixed points $1, -1$. Moreover $h(e^{i\psi}) = e^{-i\psi}$ which is identified with $e^{i\psi}$ under identification of circles in genus $2$ when $\psi = \theta, \eta, \pi-\theta, \pi-\eta$. We get $6$ fixed points, by \cite[III.7.9]{farkas_kra} this classifies the Weierstrass points of the hyperelliptic involution (See \cref{fig:g23}).

    We finish by justifying that we can take the coordinate chart of projective space where $\omega_1(z) \neq 0$, which holds since there are $2g-2 = 2$ zeroes of any holomorphic differential \cref{thm:canonicalmapisanembedding}. The zero of $\omega_1(z)$ at $z= \infty$ has order 2 since the denominator of each term contains a $z^2$ so $\omega_1(z)$ is nonzero in the fundamental domain.

    To see that the images of the Weierstrass points come in reciprocal pairs, notice $x_1 = -x_2$, $x_4 = -x_6$, and $x_3= - x_5$.  Combining this fact with \cref{lem:-zmap}, we conclude the images are indeed reciprocal values.
\end{proof}

\subsection{Non-Hyperelliptic proof}
It is left to prove \cref{thm:maing3}. 

\begin{lem}\label{lem:propertiesofG}
The group $G$ is a Fuchsian Schottky group of rank $g$. Moreover $G$ keeps invariant the unit circle and the fixed points of each $f_i$ all lie on a common circle in the following cycle configuration
\[
(A_1,B_1,A_2,B_2,\dots,A_g,B_g).
\]
Where the order to be taking is clockwise, see for example \cref{fig:g23}.
\end{lem}
\begin{proof}
The group $G$ pairs $g$ pairs of circles equally spaced about the unit circle, therefore it is Schottky. Since each $f_k$ is defined as a product of reflections in circles and lines orthogonal to the unit circle, this group is Fuchsian. Since  $\sigma$ commutes with all of $\sigma_1 $, $\sigma_2 $, and $\sigma_3 $, hence each $f_k$ (and therefore the entire group $G$) preserves the unit circle. Direct inspection shows the last part of the lemma.
\end{proof}

We will be using the notion of harmonic pairs, as introduced in \cite[Definition 1.1a]{Keen80}, which we recall here for completeness. For $x,y,u,v\in \C$ on a common line, we denote the cross-ratio as
$$(x,y; u,v) := \frac{u-x}{u-y}\cdot \frac{v-y}{v-x}.$$

\begin{defn} 
Let ${f}_i$ and ${f}_j$ be M\"obius transformations with fixed points $(A_i,B_i)$ and $(A_j,B_j)$ which are distinct. The unique pair of points $(u,v)$ in $\cp^1$ for which the cross-ratio
\[
(A_i,B_i;u,v) = (A_j,B_j;u,v) = -1
\]
is called the harmonic pair with respect to ${f}_i$ and ${f}_j$.
\end{defn}
\begin{defn}
We say a Schottky group $\Gamma$ admits a \emph{mutually harmonic set of generators} if there exists generators $f_1,\dots,f_g$ of $\Gamma$ and a pair of points $(u,v)$ which are harmonic with respect to $f_i$ and $f_j$ for every $i,j$. 
\end{defn}
 
\begin{lem}
\cite[Lemma 2]{Keen80}\label{lem:hyperellipticthenmutuallyharmonicgens} Every hyperelliptic Riemann surface is uniformized by a Schottky group with a mutually harmonic set of generators.
\end{lem}

\begin{lem}
If $g\geq 3$, then $S$ is not hyperelliptic. 
\end{lem}
\begin{proof}
Assume $S$ is hyperelliptic. By Lemma \ref{lem:hyperellipticthenmutuallyharmonicgens} there exists a Schottky group $\Gamma$ admitting a a mutually harmonic set of generators $\tilde{f}_1,\dots,\tilde{f}_g$ with respect to a pair of points $(u,v)$. Since $\Gamma$ and $G$ must be conjugate by some Mobius transformation $\phi$, therefore $\tilde{f}_i = \phi^{-1}\circ f_i \circ \phi$ for every $i=1,\dots,g$. 

In particular, if we denote by $(A_i,B_i)$ the fixed points of $f_i$, then the fixed points of $\tilde{f}_i$ are $(\phi^{-1}(p_i),\phi^{-1}(q_i))$. Since all of the fixed points of each $f_i$ lie in a common circle by Lemma \ref{lem:propertiesofG}, all the fixed points of $\tilde{f}_i$ will lie on a common circle. Notice that the only way for $\tilde{f}_1,\ldots,\tilde{f}_g$ to be a mutually harmonic set of generators implies that the pair of points must be $(u,v) = (0,\infty)$. The configuration of the fixed points given in Lemma \ref{lem:propertiesofG} cannot be mapped in that way. This gives the desired contradiction and shows that $S$ it is not hyperelliptic.
\end{proof}

\section{Explicit Computations}\label{sec:algebraiccurves}

In this Section we will first analyze the cases of complete intersections, corresponding to genus 3 and 4. We also write out how we compute explicitly the Mobius transforms and circles for the family $S_{\eta,\theta}$.

Let $S:= S_{\eta,\theta}$ be a Riemann surface of genus $g>2$ as given in \cref{def:family}, then \cref{thm:maing3} gives an embedding of $S$ in $\cp^{g-1}$ as an algebraic subvariety. Denote by $I_S\subseteq\C[x_0,\dots,x_{g-1}]$ the ideal of $\phi_S(S)$ and by $(I_S)_d$ its homogeneous component of degree $d$, this is a finite dimensional vector space. The group $D_g$ acts linearly on the space $H^0(S,\Omega_S)$ of holomorphic 1-forms. Hence induces an action in the space of monomials of degree $d$ in $x_0,\dots,x_{g-1}$ that preserve $(I_S)_d$.

\begin{lem}\label{lem:actionD_g}
    The group $D_g$ acts as follows:
\begin{enumerate}
  \item $r(z) = \exp(2\pi i/g) z $: $r^* \omega_j = \omega_{j+1}$, so we get a cyclic action of the canonical curve:
        \begin{displaymath}
          [x_0 : \cdots : x_{g-1}] \mapsto [x_1 : \cdots : x_{g-1} : x_0].
        \end{displaymath}
  \item $h(z) = 1/z$: $h^* \omega_j = -\omega_{g-j}$, so we get the action
        \begin{displaymath}
          [x_0 : x_1 : \cdots : x_{g-1}] \mapsto [x_0 : x_{g-1} : \cdots : x_1]
        \end{displaymath}
        (i.e. the first coordinate is fixed and the remaining $g-1$ coordinates are reversed; because we use homogeneous coordinates we can ignore the factor of $ -1 $ picked up in each coordinate).
\end{enumerate}
\end{lem}
\begin{proof}
    Follows by a direct computation.
\end{proof}

\subsection{Genus 3}\label{subsec:g3}
Let $S\subseteq\cp^2$ be a Riemann surface of genus $3$ from \cref{def:family}. \cref{lem:g34ideal} guarantees the ideal $I_S\subseteq \C[x_0,x_1,x_2]$ is generated by a unique polynomial $F(x_0,x_1,x_2)$ of degree 4. For every $g\in D_3$ we necessarily have that $g^*F = \lambda F$ for some non-zero constant. The orbits of the degree $4$ monomials under the action of $D_3$ are the following:
\begin{gather*}
    \{ x_0^4, x_1^4, x_2^4 \}, \{ x_0^3x_1, x_0^3 x_2, x_0 x_1^3, x_0 x_2^3,x_1^3x_2,x_1x_2^3 \}, \{ x_0^2 x_1^2, x_1^2x_2^2, x_1^2x_2^2\}, \{x_0^2x_1x_2,x_0x_1^2x_2,x_0x_1x_2^2\}
  \end{gather*}
Hence the quartic is of the form
\begin{align*}
    F(x_0,x_1,x_2) &= c_1(x_0^4 + x_1^4 + x_2^4)+c_2(x_0^3x_1 + x_0^3 x_2 +x_0x_1^3+ x_0 x_2^3 + x_1^3x_2+x_1x_2^3)\\&{}+ c_3(x_0^2 x_1^2 + x_1^2x_2^2+x_1^2x_2^2)+ c_4(x_0^2x_1x_2 + x_0x_1^2x_2+x_0x_1x_2^2).
  \end{align*}
  
\begin{rmk}
There exists a unique action up-to projective linear transformations of $D_4$ on a quartic curve in $\cp^2$, cf. \cite{KuribayashiKomiya79}. Hence the equations are common to all quartics with an action of $D_4$. Moreover, there is a change of coordinates that makes $c_2 = 0$. So this construction depends on 2 parameters. Due to the non-uniqueness of this transformation we prefer not to make this change of coordinates.
\end{rmk}

Specifically to our case are the real structure and their fixed points, this gives a way of making more precise the equation of such quartic by evaluating on very few points. 

Notice that the action of $h$ on $\cp^2$ has as fixed locus the line $x_1 = x_2$ together with the point $[0:1:-1]$. Intersection the line $x_1=x_2$ with the curve $S$ gives $4$ fixed points on the curve. Representatives of these points in $\Omega$ are given by $1,-1,\exp(\theta i)$, $\exp(\eta i)$. Hence we have
\begin{displaymath}
    \begin{array}{rr}
      \phi_S(1) = [1:u_1:u_1] & \phi_S(-1) = [1:u_2:u_2]\\
      \phi_S(e^{\theta i}) = [1:u_3:u_3] & \phi_S(e^{\eta i}) = [1:u_4:u_4]
    \end{array}
  \end{displaymath}
  \begin{equation*}
      \phi_S(1) = [1:u_1:u_1] \quad \phi_S(-1) = [1:u_2:u_2]\quad
      \phi_S(e^{\theta i}) = [1:u_3:u_3] \quad \phi_S(e^{\eta i}) = [1:u_4:u_4]
  \end{equation*}
We therefore obtain a system of equations in the variables $c_1,c_2,c_3,c_4$ by evaluating the points above on the equation $F$.

\subsection{Genus 4}\label{subsec:g4}
Let $S\subseteq\cp^3$ be a Riemann surface of genus $4$ as in \cref{def:family}. \cref{lem:g34ideal} says the ideal $I_S\subseteq \C[x_0,x_1,x_2,x_3]$ is generated by two polynomials $Q,T$ of degrees $2$ and $3$ respectively. The orbits of the degree two monomials under the action of $D_4$ are
  \begin{gather*}
    \{ x_0^2, x_1^2, x_2^2, x_3^2 \}, \{ x_0 x_1, x_1 x_2, x_2 x_3, x_3 x_0 \}, \{ x_0 x_2, x_1 x_3\}
  \end{gather*}
  hence the quadric $Q$ is of the form
  \begin{align*}
    Q(x_0,x_1,x_2,x_3) &= c_1(x_0^2 + x_1^2 + x_2^2 + x_3^2)\\&{}+c_2(x_0 x_1 + x_1 x_2 + x_2 x_3 + x_3 x_0)\\&{}+ c_3(x_0 x_2 + x_1 x_3).
  \end{align*}
  The orbits of the degree three monomials under the cyclic action $r^*$ in item (1) of \cref{lem:actionD_g} are
  \begin{gather*}
    \{ x_0^3, x_1^3, x_2^3, x_3^3 \},\\
    \{ x_0 x_1^2, x_1 x_2^2, x_2 x_3^2, x_3 x_0^2 \}, \{ x_0^2 x_1, x_1^2 x_2, x_2^2 x_3, x_3^2 x_0 \},\\
    \{ x_0^2 x_2, x_1^2 x_3, x_2^2 x_0, x_3^2 x^1 \},\\
    \{ x_0 x_1 x_2, x_1 x_2 x_3, x_2 x_3 x_0, x_3 x_0 x_1 \}.
  \end{gather*}
  The transformation $h^*$ in item (2) of \cref{lem:actionD_g} sends $ x_0 x_1^2 $ to $ x_3^2 x_0 $ and so the two orbits on the second line of the above
  display are combined into a single orbit; one sees that no other orbits combine (since $h^*$ exactly swaps $ x_1 $ and $ x_3 $ and fixes $ x_0 $ and $ x_2 $)
  so in total we have four orbits and the cubic polynomial $T$ is
  \begin{align*}
    T(x_0,x_1,x_2,x_3) &= d_1 (x_0^3 + x_1^3 + x_2^3 + x_3^3)\\
                         &+ d_2 (x_0 x_1^2 + x_1 x_2^2 + x_2 x_3^2 + x_3 x_0^2 + x_0^2 x_1 + x_1^2 x_2 + x_2^2 x_3 + x_3^2 x_0)\\
                         &+ d_3 (x_0^2 x_2 + x_1^2 x_3 + x_2^2 x_0 + x_3^2 x^1)\\
                         &+ d_4 (x_0 x_1 x_2 + x_1 x_2 x_3 + x_2 x_3 x_0 + x_3 x_0 x_1).
  \end{align*}
  This exhibits $ 2 + 3 = 5 $ parameters (generically we can divide through and assume $ c_3 = d_4 = 1 $) for the genus 4 curve; we expect $ 3g - 3 = 9 $ parameters for the
  generic genus 4 curve but of course this one is highly symmetric. From the symmetry in fact we should see only two parameters, the angles $ \theta $ and $ \eta $, with the
  parameters $ c_1 $, $ c_2 $, $ d_1 $, $ d_2 $, and $ d_3 $ functions of $ \theta $ and $ \eta $ via the uniformizing map.

  By the discussion above, we see that the fixed point locus of $h$ on $ \cp^3 $ is the plane $ x_1 = x_3 $ together with the two points $ [0:1:0:\pm 1] $.
  Intersecting the plane with the canonical curve (which is degree 6), we get six fixed points on the curve. These are all the fixed points of $h$, and the two points $ [0:1:0:\pm 1] $ do not lie on the canonical curve. The six fixed points are represented in $ \Omega $ by $1$, $-1$, $\exp(\theta i)$, $\exp(\eta i)$,
  $\exp(\pi i-\theta i)$, and $\exp(\pi i-\eta i)$. We therefore have
  \begin{displaymath}
    \begin{array}{rr}
      \phi_S(1) = [1:u_1:v_1:u_1] & \phi_S(-1) = [1:u_2:v_2:u_2]\\
      \phi_S(e^{\theta i}) = [1:u_3:v_3:u_3] & \phi_S(e^{\pi i-\theta i}) = [1:u_4:v_4:u_4]\\
      \phi_S(e^{\eta i}) = [1:u_5:v_5:u_5] & \phi_S(e^{\pi i-\eta i}) = [1:u_6:v_6:u_6]
    \end{array}
  \end{displaymath}
  where the twelve new variables $ u_k $ and $ v_k $ for $ 1 \leq k \leq 6 $ give six known points on the curve. One can now use algebra
  to fit the hypersurfaces $Q$ and $T$ onto these six known points (the system is highly over-determined): we obtain a system of seven equations
  \begin{multline*}
    0 = Q(1,u_1,v_1,u_1) = Q(1,u_2,v_2,u_2)= T(1,u_1,v_1,u_1) = T(1,u_2,v_2,u_2) = T(1,u_3,v_3,u_3)
  \end{multline*}
  in the five variables $ c_1 $, $ c_2 $, $ d_1 $, $ d_2 $, and $ d_3 $ which can be used to write down the canonical curve based only on evaluating the canonical map at five general points.

  \subsection{Explicit circles and mappings} \label{sec:circles+maps}

In this section we choose a fixed matrix to map between two isometric circles, and compute the fixed points. To map one isometric circle onto another circle, we are able to choose the amount of rotation. We will choose the amount of rotation so that the matrices have positive real trace, which will be in fact strictly larger than $2$. This can be seen by using the characterization of the action of M\"obius transformations on the complex plane up to conjugation by their fixed points, or equivalently their trace \cite[Chapter I]{maskit}. Specifically given two isometric circles we identify them with the following lemma.
\begin{lem}\label{lem:circtomatrix}
    Consider two disjoint circles in $\C$ centered at $z$ and $w$ with the same radius $r>0$, $C = S(z,r)$ and $C' = S(w,r)$. Suppose f is defined by
    $$ f= \begin{bmatrix} a & b \\ c & d\end{bmatrix}\quad \text{ for }\quad a = \frac{w}{r} e^{i\theta}\quad d = -\frac{z}{r} e^{i\theta} \quad c = \frac{1}{r} e^{i\theta} \quad b = \frac{ad - 1}{c} = -\frac{wz e^{2 i \theta}+r^2}{r e^{i\theta}} $$
    satisfying $f(C) = C'$ and $f(B(z,r)) = \C\setminus B(w,r)$. In order for the trace to be real and positive, we set
    $$\theta_0 = -\arg(w - z). $$

    The two fixed points are given by the attracting, respectively repelling fixed points $A$ and $B$ where
    $$A = \frac{a-d +\sqrt{(a-d)^2 + 4bc}}{2c}\quad B =  \frac{a-d - \sqrt{(a-d)^2 + 4bc}}{2c}.$$
\end{lem}

\begin{proof}[Proof of \cref{lem:circtomatrix}]
    We first show $f(C) = C'$. Let $p \in \C$ so $|p-z|= r$. We compute
    $$|f(p) - w| = \left|\frac{\frac{e^{i\theta}}{r} wp - \frac{wze^{2i\theta}+r^2}{re^{i\theta}}}{\frac{e^{i\theta}}{r} p - \frac{z}{r}e^{i\theta}} -w\right| = \left|\frac{e^{2i\theta} w(p - z) - r^2}{e^{2i\theta}(p - z)} - w\right| = r,$$
    which is sufficient since $f$ is invertible. Since $f$ preserves the circle, we only need to check one point in the interior. Notice that $f(z) = \infty$, so $f(B(z,r)) = \C\setminus B(w,r).$ The choice of $\theta_0$ guarantees the trace is given by 
    $$a+d = \frac{e^{i\theta_0}}{r}(w-z) = \frac{1}{r} |z-w| \in \R_{> 0}.$$ 

    Finally to compute the fixed points, the fixed points satisfy
    $$az + b = z(cz+d) \Leftrightarrow c z^2 + (d-a)z - b = 0,$$
    where the solution comes from the quadratic formula and the fact that circles are disjoint implies the discriminant is nonzero.
\end{proof}

\begin{prop}
    For the real curves $\sC_i = \sC(z_i,r)$ to $\sC_i' = \sC(w_i,r)$ given in \cref{def:family} we have the following formulas:
   $$z_1 = \exp\left(i\frac{\theta+\eta}{2}\right)\sec\left(\frac{\theta-\eta}{2}\right) \quad r = \tan\left(\frac{\theta-\eta}{2}\right),$$
   and for $j=2,\ldots, g$
   $$w_1 = \overline{z_1}\quad z_j = e^{\frac{2\pi i}{g} (j-1)}z_1 \quad w_j = e^{\frac{2\pi i}{g} (j-1)} w_1.$$
\end{prop}
\begin{proof}
    We first notice that the four circles are all orthogonal to the unit circle, so reflection across the unit circle provides a real structure for the curve. 

    Next we will focus on constructing the center and radius of $\sC_1$. Recall $\sC_1$ is the circle orthogonal to the unit circle $\sS^1$ which goes through the points $e^{i\eta}$ and $e^{i\theta}.$ In order to find the Euclidean center, we take the intersection point of the tangent lines to the unit circle at $e^{i\theta}$ and $e^{i\eta}$, respectively.
    The tangent lines are given by finding the lines orthogonal to the radial line through $e^{i\theta}$ and $e^{i\eta}$
$$y- \sin(\eta) = -\frac{\cos(\eta)}{\sin(\eta)} (x- \cos(\eta)) \quad y- \sin(\theta) = -\frac{\cos(\theta)}{\sin(\theta)}(x- \cos(\theta)).$$
So the center of the circle is given by 
$$z_1 = \cos\left(\frac{\theta + \eta }{2}\right)\sec\left(\frac{\theta - \eta}{2}\right) + i \sin\left(\frac{\theta + \eta}{2}\right) \sec\left(\frac{\theta- \eta}{2}\right)$$
with radius found by taking the distance from $z_1$ to $e^{i\eta}$:
$$r_1 = \sqrt{(\cos(\eta) - \Re(z_1))^2 + (\sin(\eta) - \Im(z_1))^2} = \tan\left(\frac{\theta - \eta}{2}\right).$$
Since all the circles have the same radius, we have now found the correct radius for the six circles. The other centers are given by applying the correct complex conjugation to get $w_1$, and then rotating by $2\pi/g$. The fact that we get a Schottky group and thus a Riemann surface comes from \cref{prop:schottkycondition}.
\end{proof}
\section{Numerical Experiments} \label{sec:NumericalExperiments}
In this section we first share information about the algorithm and code used for making the experiments. We next include computations for genus 2 where we give approximate Weierstrass points and then in genus $3$ where we approximate the quartic defining the curve by looking at the real points.

\subsection{An algorithm for generating elements of the free group}  \Cref{alg:circ2schottkywords} inputs pairs of isometric circles and outputs the elements in the Schottky group up to level $L$. This algorithm minimizes the amount of matrix multiplication needed in order to store elements of the Schottky group. The key idea is when generating a word, we also add a ``tag'' which keeps track of what generator is on the right side of the word. This idea of keeping track of the tags was given in \cite[p.115]{indras_pearls} and this algorithm is the same but easier to read since it uses capabilities of modern programming languages. Keeping track of the tags is also the key idea to detecting elements of the coset representatives $G/G_n$, so the remaining implementation for evaluating the Poincare series at points on the Riemann surface is elementary. See \cite{mathrepo} for an implementation in Julia \cite{julia}, which is how the following computations were generated.
\begin{algorithm}[hbt]
\caption{From circles to words in Schottky groups}\label{alg:circ2schottkywords}
\begin{algorithmic}[1]
    \State Input $g$ centers $z_1,\ldots, z_g$
    \State Input $g$ centers $w_1,\ldots, w_g$
    \State Input $g$ radii $r_1,\ldots, r_g >0$
    \State Input length of words generated $L$.
    \For{$i=1,\ldots, g$}
        \State $f_i = \begin{bmatrix} \frac{w_i}{r_i} & - \frac{w_i z_i}{r_i} - r_i\\ \frac{1}{r_i} & \frac{-z_i}{r_i}
        \end{bmatrix}$ \Comment{as in \cref{lem:circtomatrix}}
    \EndFor
    \Procedure{MatrixWords}{$g, L$} \Comment{Construct words in Schottky group of length $\leq L$}
    \State $list = [(Id_2,0)]$ \Comment{$Id_2$ is $2\times 2$ identity matrix, tag is $0$}
    \For{$level = 1,\ldots, L$}
        \State $oldlist \gets list$ \Comment{$oldlist$ has all words length $<level$}
        \For{$w \in oldlist$}
            \If{$w[2] \neq -j$}{ add $(w[1]\cdot f_j,j)$ to $list$}
            \EndIf
            \If{$w[2] \neq j$}{ add $(w[1]\cdot f_j^{-1},-j)$ to $list$}
            \EndIf
        \EndFor
    \EndFor
    \State return [$w[1]$ for $w$ in $list$]
    \EndProcedure
\end{algorithmic}
\end{algorithm}
\subsection{Numerical Weierstrass Points Genus 2}\label{sec:g2numerical}
We consider $\eta = \frac{\pi}{12}$ and $\theta= \frac{\pi}{4}$ which were used to sketch \cref{fig:g23}. We estimate the equation of the curve is
$$y^2 = \prod_{j=1}^6 (x-\alpha_j)$$
where $\alpha_j \approx\tilde{\phi}_S(x_j)$ calculated for words of length $L\leq 12$. The values are approximated as follows, where the imaginary components are dropped as they all appear on the order of $10^{-11}$
\begin{align*}
    \alpha_1 &= -11.47448708745693& \alpha_3 &=   4.084600885787089
& \alpha_4 &= -11.787692546430492
  \\
   \alpha_2 &= -0.08714986494625039 &
   \alpha_5 &= 0.2448219613030484 &
 \alpha_6 &= -0.08483424521467196 
\end{align*}
We notice that the real structure can be seen here as all of the imaginary components of the roots are zero. We also verify the reciprocal values of the roots are given by
\[\alpha_1 \alpha_2 = 0.9999999999993655,
\;\alpha_3 \alpha_5 = 0.9999999999985639,\; \alpha_4 \alpha_6 = 0.9999999999990453. \]

Finally, we plot the curve for real values of $x$ and $y$ in \cref{fig:g2example}.
\begin{figure}[ht]
    \centering
    \includegraphics[width=.3\textwidth]{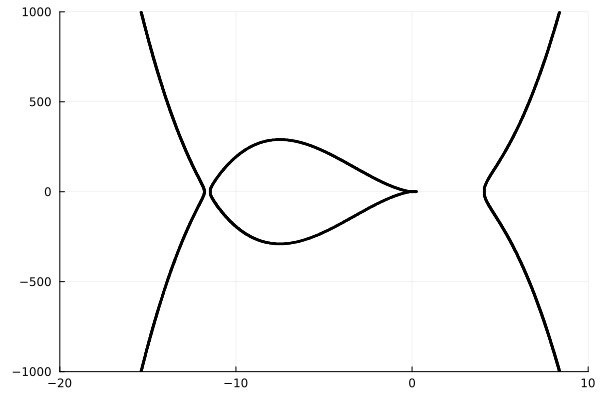}\hfill
    \includegraphics[width=.3\textwidth]{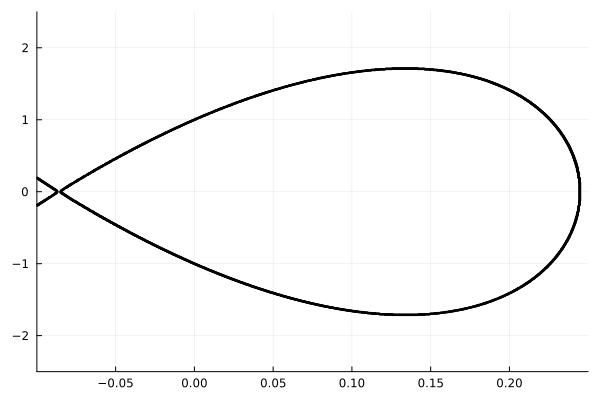}\hfill
    \includegraphics[width=.3\textwidth]{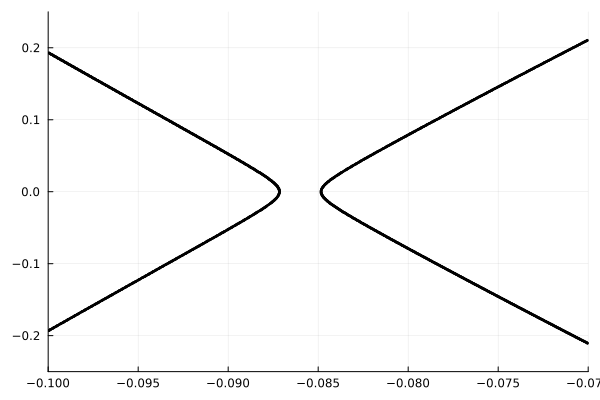}
    \caption{A representation of the real points of the genus 2 curve with $\eta = \frac{\pi}{12}$ and $\theta = \frac{\pi}{4}$ on the left. In the middle we zoom in near $x=0$, and on the right, we zoom further in near $x=-.09$ to see these curves create 3 separate smooth curves. }
    \label{fig:g2example}
\end{figure}

\subsection{Approximating the quartic in genus three}\label{sec:g3numerical}

To experiment with genus $3$, we decided to sample along the real points of the $M$-curve, and approximate the quartic to fit the real data. That is we did the following.
\begin{enumerate}
    \item Fix $0<\eta<\theta < \frac{\pi}{3}$ and choose $N, L \geq 1$.
    \item Sample $N$ points uniformly on the unit circle and compute the approximate image in $\cp^2$ by truncating the Poincar\'e theta series of \cref{prop:oneforms} to words of length at most $L$.
    \item Plot the real points in the plane by taking the chart in $\cp^2$ where the third component is $1$. We also truncate to the real parts since the imaginary parts are small. This is used to form the scatter plots in \cref{fig:plotg3_1} and \cref{fig:plotg3_2}.
    \item Now we add a contour plot of the level sets by plotting the real part of our approximate polynomial $p(z_1,z_2,z_3)$ at $p(x,y,1)$. To obtain the approximate polynomial, we solve a homogeneous least squares problems $Ax=0$ subject to the constraint $\norm{x} = 1$, where $x$ is a basis of monomials, and $A$ is the $N\times 15$ matrix formed by evaluating the monomials at the $N$ sample points.
\end{enumerate}

\begin{ex} \label{example:g3}
Consider $\eta = \pi/12$, $\theta = \pi/4$, $N= 50$ and $L=8$. We obtained the following equation, where the imaginary parts are on the order of $10^{-8}$ and we keep only the first 6 decimal places
\begin{equation*}
    \begin{split}
        p(z_1,z_2,z_3)& =-0.006341 z_1^4  -0.006347 z_2^4 - 0.0063523 z_3^4-0.039327 z_1^3 z_2 - 0.039291 z_1^3 z_3 \\
        &- 0.039302 z_1 z_2^3- 0.298144 z_1^2 z_3^2 - 0.039225 z_2 z_3^3 - 0.039441  z_1 z_3^3\\
        &-0.298246 z_1^2 z_2^2 - 0.298398 z_2^2 z_3^2  - 0.039336 z_2^3 z_3\\
        & +0.490895 z_1^2 z_2 z_3  + 0.491323 z_1 z_2^2 z_3 +0.491268 z_1z_2 z_3^2 .
    \end{split}
\end{equation*}

Notice that the Equation closely follows the symmetry of the curve expected from \cref{subsec:g3} by only sampling 50 points and truncating the series after words of length $8$. In \cref{sec:finalconverge} we discuss potential options and difficulties for obtaining better estimates.

See also \cref{fig:plotg3_1} for a plot of the sampled points and a contour plot of the approximate polynomial.

\begin{figure}[htb]
\includegraphics[width=.45\textwidth]{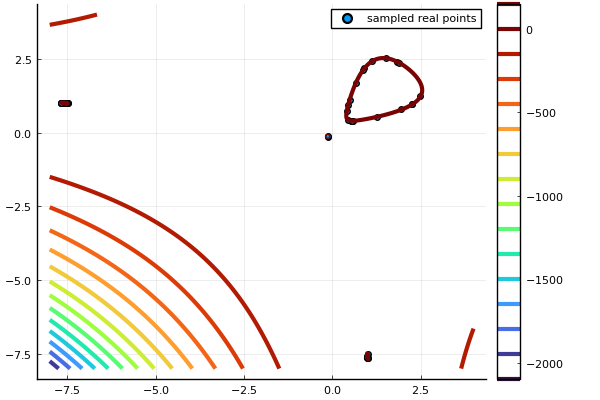}
\includegraphics[width = .45
\textwidth]{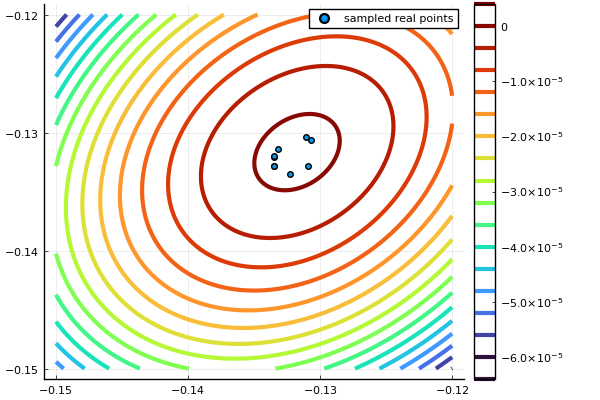}
    \caption{From \cref{example:g3} we plot uniformly sampled real points in blue alongside a contour plot of the approximate polynomial. Notice that there is red at each of the 3 dots, as well as the one loop we can see which which correspond to the 4 ovals making this an $M$-curve. On the right we zoom in very close to $(0,0)$, where the quartic in fact detects this very small oval as well.}
    \label{fig:plotg3_1}
\end{figure}
\end{ex}

We remark that (3) could generate issues as the Poincare theta series do not converge at the fixed points $A_n,B_n$ which live on the unit circle. Naturally we can first just hope this doesn't happen since they form a set of measure zero when sampling on the unit circle. Alternatively to guarantee convergence, on could sample uniformly on each of the intervals of the unit circle which are exterior to the circles. By doing this it is also nice to see from \cref{fig:g23} that the small intervals between $\sC_j$ and $\sC_j'$ form the smaller 3 circles in \cref{fig:plotg3_1}, whereas the three intervals between $\sC_1, \sC_2'$ and $\sC_2, \sC_3'$ and $\sC_3, \sC_1'$ combine to form the one larger oval.

Next we remark that in (4) the least squares solution to the homogeneous solution of equations (see \cite{HomoLS}) requires taking the eigenvector associated to the smallest eigenvalue of $A^T A$ up to a normalization. However in our experiments our smallest eigenvalue was numerically 0, so instead we could just project onto the null space by computing $(I_{15} - A^+ A) \vec{1}$ where $A^+$ is the Moore--Penrose pseudoinverse, and $\vec{1}$ is the vector of all $1$s. Note $I_{15}-A^+ A$ is the orthogonal projection onto the null space, so the choice of $\vec{1}$ is arbitrary and we can just normalize the vector to norm $1$.

\begin{figure}[hbt]\includegraphics[width=.45\textwidth]{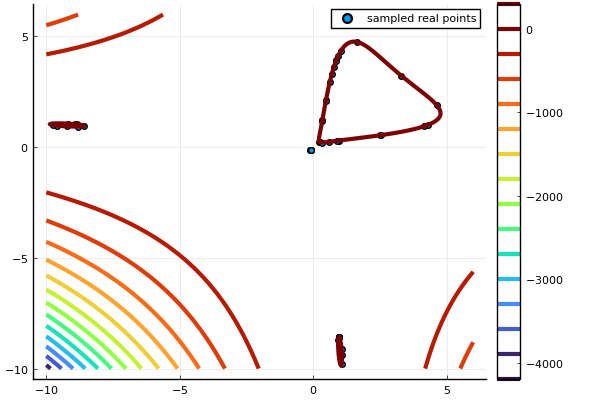}
\includegraphics[width = .45
\textwidth]{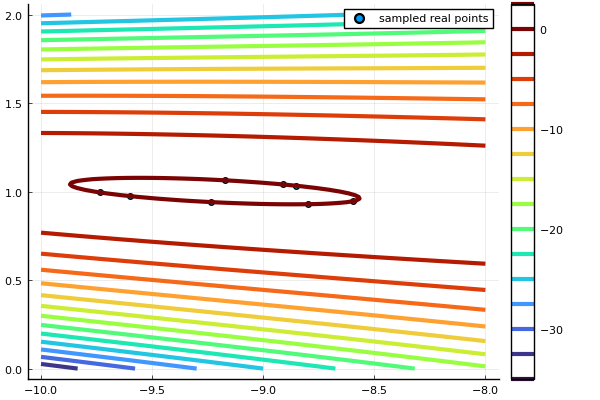}
    \caption{In this case we consider $\eta = \pi/9$ and $\theta = 2\pi/9$ $N=50$ and $L=8$, in order to see a more symmetric situation in the real ovals. The blue points are sampled uniformly along the unit circle, and the contour map overlaid represents the polynomial approximation. On the right image we zoom into the far left oval to see the approximation polynomial matches very well to the given real points.}
     \label{fig:plotg3_2}
\end{figure}

Notice that we need $N\geq 15$ in order to have enough sample points on the curve, but we chose $N=50$ since it was small enough to regularly show all four ovals while still leaving very short computational times of at most 1 minute. We expect we would obtain much more accurate information using more refined methods of learning varieties from curves \cite{BKSW2018}. The remarkable part here is that our simple method with a low level of accuracy allows us to find a quartic that can see all 4 ovals.

\subsection{Computing the Riemann Matrix}\label{sec:numericalRiemann}

In this section, we want to use the following theorem to approximate the Riemann Matrix for $g=3$. To do this, for each $n =1,\ldots, g$, define $\mu_n$ so that
$$\frac{f_n(z) - A_n}{f_n(z) - B_n} =\mu_n  \frac{z-A_n}{z-B_n}, \quad |\mu_n| < 1,$$
in which case $A_n$ is the attracting and $B_n$ the repelling fixed point.
Namely if $f$ is our fixed transformation from \cref{lem:circtomatrix}, the associated value of $\mu$ is given by
$$\mu =  \frac{a-Ac}{a-Bc}$$

We fix a basis of homology $a_j,b_j$ by setting $a_j$ to follow the boundary of $\sC_j$ counterclockwise, and $b_j$ to be the closed loop connecting $\sC_j$ to $\sC_j'$.

\begin{thm}\cite[Theorem 5.1]{Algebro_geometric}
    The Riemann matrix normalized to this basis of differentials is given by 
    \begin{equation}\label{eq:Rmatrix}
    R_{nm} = \frac{1}{2\pi i}\begin{cases}\sum_{f \in G_m\backslash G/G_n} \log\{A_m,B_m, f A_n, fB_n\} &m\neq n\\
    \log \mu_n + \sum_{f\in G_m\backslash G/G_n, f\neq I} \log\{A_m,B_m, f A_n, fB_n\} &m=n
    \end{cases}\end{equation}
    where the curly brackets are the cross-ratio
    $$\{z_1,z_2,z_3,z_4\} = (z_1-z_3)(z_2-z_4)(z_1-z_4)^{-1}(z_2-z_3)^{-1}.$$
\end{thm}
We refer the reader to \cite{Schmies2011} for further discussion on computing rates of convergence. For our purposes we wanted to see the fact that we indeed have $M$-curves by having a purely imaginary Riemann matrix. Indeed for $\eta = \pi/12$ and $\theta = \pi/4$ generating words of length at most $10$, we approximated the Riemann matrix of the genus 3 curve by
$$i\begin{bmatrix}
    0.39980895499614727 & -0.004497125148067202 &-0.004497125148229181\\
    -0.004497125148067202 &0.3998089549956627 &-0.004497125149408407\\
    -0.004497125148229181 &-0.004497125149408407& 0.3998089549963685
\end{bmatrix},$$
where the real part was on the order of $10^{-12}$. To obtain the decimal point accuracy for this computation we refer the reader to \cref{sec:finalconverge} and \cite[Lemma 4.3.3]{Schmies2011}.

\section{Final discussion} \label{sec:finaldiscussion}

In this paper we have focused on the description of algebraic curves for small genus. Our framework however is more general and works in any genus. In particular \cite{BKSW2018} provides a source for general strategies to recover algebraic curves from point sets. We also conclude with a few words on quantifying convergence, other families of curves, and moduli.


\subsection{Quantifying convergence} \label{sec:finalconverge} In order to obtain estimates on numerical accuracy we know of two methods: one away from compact sets and the other for different truncation options.

In the first case, we quantify the remainder by considering the summation of the Poincar\'e theta series truncated by word length of $f$ with $|f|$. The next Lemma is essentially proven in \cite[Lemma~6.1]{Bogatyrev}, which only deals with hyperelliptic curves, but in fact works for any Fuchsian Schottky groups by applying the transformation $T$ mapping the unit circle to the real line.

\begin{lem}(Bogatyrev)
    We have
    $$\sum_{|f|>k\, f\in G/G_n} \left|\frac{1}{z-f(B_n)} - \frac{1}{z-f(A_n)}\right| \leq \left[\frac{1}{\mathrm{dist}(\mathbf{K}, \Lambda(G))^2} + o(1)\right](\sqrt{\mathbf{q}}+1) \left(\frac{\sqrt{\mathbf{q}}-1}{\sqrt{\mathbf{q}}+1}\right)^k $$
    where $z\in \mathbf{K}$ is a compact subset of $\Omega(G)$ and $o(1)\to 0$ as $k\to\infty$ and $\mathbf{q}$ is a function of $\theta$ and $\eta$ and $g$, and can be made explicit in terms of cross-ratios of specific points on the circle.
\end{lem}
\begin{proof}
The proof of \cite[Lemma~6.1]{Bogatyrev} applies line by line, since with the M\"obius transformation $T$ our circles are all orthogonal to the real line. M\"obius transformations preserve cross-ratios, hence the estimates in loc.cit. are the same.
\end{proof}
Now the key idea is that once we fix $\eta, \theta$, the rate of convergence is a function of the distance of our sample points from the real line. Though we might apriori be missing key geometries of the curve by sampling so far away from the limit set, but since the structure is algebraic, knowing even a little bit of the curve should give us all the information we need. To obtain explicit numerical information, the $o(1)$ quantity must be made explicit.

In the second case, we refer the reader to \cite{Schmies2011}, but provide a specific case for comparison. In this case instead of truncating series via word length, the author allows for a more general truncation option via any subset $S$ of $G/G_n$ which is \textit{suffix closed}, which captures the idea that $\tilde{f} \cdot f\in S$ should also have $f\in S$, but for convenience we will present the idea of convergence for $S$ to be all words in $G/G_n$ of length less than $k$, and set $F$ to be a fixed fundamental domain of $S$ in $\C$. Set $|\gamma_f|$ to be the modulus of the bottom left entry the $2x2$ matrix representing $f$, which controls the radius of the isometric circles. We will consider a function $\kappa_k$ for $k\in N$ which is very close to $\displaystyle{\min_{\sigma\in G, |\sigma| =k} \abs{ \frac{\gamma_\sigma}{\gamma_{\hat{\sigma}}}}}$ (\cite[Lemma 4.2.8]{Schmies2011}), where $\hat{\sigma}$ is the word of length $k-1$ formed by removing the right-most generator in the word. This captures the successive ratios of the circle radii when applying generators. If there is some $k$ large enough so that $2g-1< \kappa_k^{-2}$ then there is guaranteed convergence of the series \cite[Theorem 4.2.10]{Schmies2011}.

We are now prepared to quantify convergence.

\begin{lem}[Lemma 4.3.4 \cite{Schmies2011}]
    Fix $k \geq 1$. For $k$ large enough
    \[
    \begin{split}&\abs{\sum_{|f|>k, f \in G/G_n} \frac{1}{z-f(B_n)} - \frac{1}{z-f(A_n)}} < \left(\min_{\sigma \in G, |\sigma| = k}\min_{z\in F} |z-\sigma(z)|\right)^{-2}\cdot \\
    &\left( \frac{(1-\kappa_{k}^{-2})|A_n - B_n|}{(1+\kappa_{k}^{-2})(1-(2g-1)\kappa_{k}^{-2})} \right)\sum_{|f| = k} \frac{|\gamma_f|^{-2}}{\min_{z\in F} |A_n - f^{-1}(z)|\cdot\min_{z\in F} |B_n - f^{-1}(z)| } .
    \end{split}
    \]
\end{lem}

We remark that the benefit of this lemma is the lack of the $o(1)$ term that needs quantified. However unlike the first case where the cross ratios are simple to compute, in this case we need to compute minimum distances to the fundamental domain over many points, which is computationally expensive to do explicitly. Moreover, here we can see the benefit of choosing different subsets of $S$ instead of taking all words of length up to $k$ as for large $k$ this is a very large set we need to sum over in the computation of the error term as well as $\kappa_k$. For more discussion on computation see \cite[Section 4.4]{Schmies2011}.

In conclusion there do seem to be some methods for computing error terms, but they are quite involved and require additional computational tools beyond the goal of this paper.

\subsection{Other families}
We provide some references that would allow the interested reader to consider examples beyond the family of curves presented in \cref{def:family}. First, if the reader is interested in $M$-curves, then the Poincare theta series always converge and the Schottky group is Fuchsian, and the family of such curves can be parametrized \cite[Section 1.8.3]{BK13}. In general, we can construct an open subset in $\C^{3n}$ where all of the matrices satisfy the Schottky condition \cite[Theorem 31]{BK13}, which provides a sufficient condition for absolute convergence of the Poincare theta series. The following family give examples where we have circles sufficiently far apart to guarantee convergence. 
\begin{ex} Fix $g\geq 2$. For each $i =1,\ldots, g$ define transformations $f_i=\begin{bmatrix} a_i & b_i \\ c_i & d_i \end{bmatrix}$ with the following conditions.
\begin{enumerate}
  \item For all $ 1 \leq i < j \leq g $, to ensure the isometric circles of $ f_i $ and $ f_j $ are disjoint:
        \begin{equation*}
          \min\left\{\abs{\frac{a_i}{c_i} - \frac{a_j}{c_j} } , 
          \abs{\frac{a_i}{c_i} + \frac{d_j}{c_j} } ,
          \abs{\frac{d_i}{c_i} - \frac{d_j}{c_j} }\right\} > \frac{1}{\abs{c_i}} + \frac{1}{\abs{c_j}}
        \end{equation*}
  \item For all $ 1\leq i\leq g $, to ensure that the two isometric circles of $\sC_i, \sC_i'$ are disjoint:
        \begin{gather*}
          \abs{\frac{a_i}{c_i} + \frac{d_i}{c_i} } > \frac{2}{\abs{c_i}} \iff  \abs{a_i + d_i} > 2.
        \end{gather*}
\end{enumerate}
From (2) we get loxodromicity of the generators for free and only the $ 3 \binom{n}{2} + n = \frac{1}{2}(3n^2 - n) $ conditions listed need to be satisfied for the group
to be Schottky. 

\end{ex}

\subsection{A few words about moduli} \label{sec:finalmoduli}
A (not necessarily classical) Schottky group can be characterized as a discrete subgroup of $\mathrm{PSL}_2(\C)$ freely generated by purely loxodromic elements \cite{Maskit1967}. It is known that there exists non-classical Schottky groups, see \cite{Marden1974} and references therein.

Following \cite{Bers1975} we define the the Schottky space, denoted by $\mathcal{S}_g$, as the space of sets of $g$ elements in $\mathrm{PSL}_2(\C)$ that generate a Schottky group up to M\"obius transformations. In order to define a holomorphic structure on $\mathcal{S}_g$, it can be proved that $\mathcal{S}_g$ is the quotient of the \emph{Teichm\"uller space} $\mathcal{T}_g$ by a (not normal) subgroup. Fix $G$ a Schottky group; then locally around $G$ the space $\mathcal{S}_g$ parametrizes holomorphic deformations of the coefficients of generators of $G$: a choice of generators for $G$ is an element $Z_0 \in (\PSL(2,\C))^g$, and a holomorphic deformation of $ Z_0 $ is a holomorphic map $\mu: \Delta \to (\PSL(2,\C))^g$
such that $ \mu(0) = Z_0 $. Recall the following.

\begin{thm}
Every algebraic curve $C$ of genus $g>1$ is uniformized by a Schottky group of rank $g$, and these Schottky groups lie in an open set $ U \subseteq \mathcal{S}_g$ with natural coordinates (entries of matrices generating the group) which is locally the neighbourhood of that curve $C$ in $\mathcal{M}_g$.    
\end{thm}

Thus one advantage of studying Schottky groups as models for algebraic curves is the ease of construction: one can immediately write down curves of arbitrary genus (at least numerically, as we studied in earlier sections of this paper) which is a very hard thing to do with purely algebraic machinery \cite[\S 6F]{HarrisMorrison98}.

We should point out however that the matrix entries do not provide a nice coordinate system in any geometric sense: they are very local and do not represent any geometric quantities: one natural choice of coordinates are the Fenchel–Nielsen coordinates, which can be generalised to the Kleinian group setting, but the relationship between these coordinates and the matrices generating the uniformizing groups are very subtle. 

A promising research direction comes from the theory of compactifications of moduli spaces. The extended Schottky space (cf. \cite{GerritzenHerrlich1988}) parametrizes mild degenerations of Riemann surfaces. This in turn correspond to stable curves, in the sense of Deligne-Mumford. An interesting question will be to find the correct degeneration of our Schottky groups, which correspond to degenerations of the configurations of circles on the plane, to get some canonical embedded stable curves, for which the ideal might be easier to compute.

From the theoretical point of view there are several follow-up questions related with the family introduced in \cref{def:family}. A natural question is the field of definition of our family in \cref{def:family}. For example, we do not know if for $\eta,\theta$ rational, the corresponding algebraic curve $S_{\eta,\theta}$ in $\cp^{g-1}$ will be defined over $\Q$. Lastly we mention that our family defines a map $\psi: U\to \mathcal{M}_g$ depending on two real parameters $\eta,\theta$. Hence makes sense to study its regularity. Although continuity of $\psi$ is immediate, we do not know if, for example, $\psi$ is differentiable.

\bibliographystyle{alpha}
\bibliography{Sources}

\end{document}